\documentclass{article}

\usepackage{mathtools}

\usepackage{amssymb}
\usepackage{amsthm}
\usepackage[justification=centering]{caption}
\usepackage{booktabs}
\usepackage{enumitem}
\usepackage{float}
\usepackage[margin=1.25in]{geometry}
\usepackage{gnuplottex}
\usepackage[below]{placeins}

\usepackage[
      backend=bibtex
    , style=numeric
    , maxbibnames=99
    , maxcitenames=99
    ]{biblatex}
\DeclareFieldFormat[misc]{title}{\mkbibquote{#1}}

\usepackage{tikz}
\usetikzlibrary{cd}
\usetikzlibrary{calc}
\usetikzlibrary{positioning}

\usepackage[ruled, procnumbered]{algorithm2e}
\SetKw{Continue}{continue}
\SetKw{In}{in}
\SetKw{None}{None}
\SetKw{Yield}{yield}
\SetKw{And}{and}
\SetKw{Break}{break}
\SetKwInput{Input}{Input}
\SetKwInput{Base}{Base case}
\SetKwProg{Fn}{fn}{}{}

\usepackage{listings}
\lstdefinelanguage{Rust}{%
    morekeywords={bool, const, fn, for, if, in, let, mut, u64, while},
}

\newcommand{\precdot}{\prec\mathrel{\mkern-3mu}\mathrel{\cdot}}

\DeclareMathOperator{\rot}{rot}
\DeclareMathOperator{\ord}{ord}

\newtheorem{theorem}{Theorem}
\newtheorem{lemma}{Lemma}
\newtheorem{corollary}{Corollary}
\newtheorem{conjecture}{Conjecture}

\theoremstyle{definition}
\newtheorem{definition}{Definition}
\newtheorem{algorithm-thm}[algocf]{Algorithm}
\newtheorem{strategy}{Strategy}

\theoremstyle{remark}
\newtheorem{remark}{Remark}

\title{An almost linear time algorithm testing whether the Markoff graph modulo~$p$ is connected}
\author{Colby Austin Brown}
\date{}

\begin{document}
\maketitle

\begin{abstract}
    The Markoff graph modulo~$p$ is known to be connected for all but finitely many primes~$p$ (see~\citeauthor{fuchs} [arxiv:\citefield{fuchs}{eprint}]), and it is conjectured that these graphs are connected for all primes.
    In this paper, we provide an algorithmic realization of the process introduced by Bourgain, Gamburd, and Sarnak [arxiv:\citefield{bgs}{eprint}] to test whether the Markoff graph modulo~$p$ is connected for arbitrary primes.
    Our algorithm runs in~$o(p^{1 + \epsilon})$ time for every~$\epsilon > 0$.
    We demonstrate this algorithm by confirming that the Markoff graph modulo~$p$ is connected for all primes less than one million.
\end{abstract}

\section{Introduction}

The Markoff equation is
\begin{align}
    x^2 + y^2 + z^2 - xyz = 0. \label{eq:markoff}
\end{align}
It is often written as
\begin{align*}
    x^2 + y^2 + z^2 - 3xyz = 0,
\end{align*}
but the solutions to the latter are in bijection with the solutions to the former, reduced by a factor of~$3$.
In this paper, ``the Markoff equation'' will always refer to~\eqref{eq:markoff}.

If~$(a, b, c)$ is an integer solution to the Markoff equation, then it is called a Markoff triple.
Each of~$a$,~$b$, and~$c$ is a Markoff number; the set of all positive Markoff numbers is~$\mathbb{M}$.
The first few positive Markoff numbers are~$\mathbb{M} = \{3, 6, 15, 39, 87, \ldots\}$.
The set of Markoff triples is closed under taking the negation of exactly two coordinates; in this paper, we will only consider the nonnegative solutions, denoted~$\mathcal{M}$.

Fixing the values of~$y$ and~$z$, the Markoff equation is quadratic in~$x$.
The values of~$x$ which form a Markoff triple~$(x, y, z) \in \mathcal{M}$ are given by the quadratic equation
\begin{align}
        x = \frac{1}{2}\left(yz \pm \sqrt{y^2z^2 - 4(y^2 + z^2)}\right). \label{eq:quadratic}
\end{align}

These two solutions are mapped to each other via~$x \mapsto yz - x$.
The \textit{Markoff Vieta involutions} are this map applied to one coordinate of a Markoff triple:
\begin{align*}
    V_1 &\colon (a, b, c) \mapsto (bc - a, b, c), \\
    V_2 &\colon (a, b, c) \mapsto (a, ac - b, c), \\
    V_3 &\colon (a, b, c) \mapsto (a, b, ab - c).
\end{align*}

Let~$\Gamma$ be the group generated by the Markoff Vieta involutions and the permutations of~$x$,~$y$, and~$z$.
The group~$\Gamma$ has exactly one fixed point,~$(0, 0, 0)$.
Since the Markoff equation is invariant under permutations of~$x$,~$y$, and~$z$, the triples~$\mathcal{M}$ are fixed by~$\Gamma$.
By a result of Markoff, the subset~$\mathcal{M}^\times = \mathcal{M} \setminus \{(0, 0, 0)\}$ is also $\Gamma$-transitive~\cite{markoff}.
We represent this as a tree~$\mathcal{G}^\times$ with vertex set~$\mathcal{M}^\times$, rooted at $(3, 3, 3)$, and an edge between~$v_1$ and~$v_2$ labeled~$V_i$ if~$V_i(v_1) = v_2$.
(An equivalent construction has the edges directed, but since~$V_i$ is an involution, edges will always come in~$(v_1, v_2)$,~$(v_2, v_1)$ pairs.)
This tree is called the Markoff tree.
See Figure~\ref{fig:markoff-tree} for a depiction of~$\mathcal{G}^\times$.

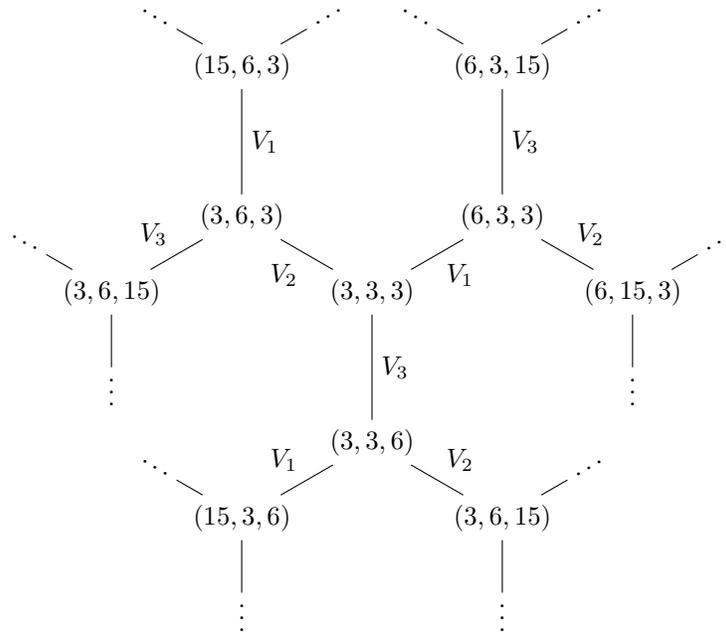
\begin{figure}
    \begin{center}
        \begin{tikzpicture}
            \draw node (root) {$(3, 3, 3)$};

            \draw ($(root) + (30:2)$) node (V1) {$(6, 3, 3)$};
            \draw ($(root) + (150:2)$) node (V2) {$(3, 6, 3)$};
            \draw ($(root) + (270:2)$) node (V3) {$(3, 3, 6)$};

            \draw (root) -- (V1) node [midway, below right] {$V_1$};
            \draw (root) -- (V2) node [midway, below left] {$V_2$};
            \draw (root) -- (V3) node [midway, right] {$V_3$};

            \draw ($(V1) + (-30:2)$) node (V12) {$(6, 15, 3)$};
            \draw ($(V1) + (90:2)$) node (V13) {$(6, 3, 15)$};
            \draw ($(V2) + (90:2)$) node (V21) {$(15, 6, 3)$};
            \draw ($(V2) + (210:2)$) node (V23) {$(3, 6, 15)$};
            \draw ($(V3) + (210:2)$) node (V31) {$(15, 3, 6)$};
            \draw ($(V3) + (-30:2)$) node (V32) {$(3, 6, 15)$};

            \draw (V1) -- (V12) node [midway, above right] {$V_2$};
            \draw (V1) -- (V13) node [midway, right] {$V_3$};
            \draw (V2) -- (V21) node [midway, right] {$V_1$};
            \draw (V2) -- (V23) node [midway, above left] {$V_3$};
            \draw (V3) -- (V31) node [midway, above left] {$V_1$};
            \draw (V3) -- (V32) node [midway, above right] {$V_2$};

            \begin{scope}[anchor=west]
                \draw (V12) -- ($(V12) + (30:1)$) node[rotate=30] {$\dots$};
                \draw (V12) -- ($(V12) + (270:1)$) node[rotate=270] {$\dots$};
                \draw (V13) -- ($(V13) + (30:1)$) node[rotate=30] {$\dots$};
                \draw (V13) -- ($(V13) + (150:1)$) node[rotate=150] {$\dots$};
                \draw (V21) -- ($(V21) + (30:1)$) node[rotate=30] {$\dots$};
                \draw (V21) -- ($(V21) + (150:1)$) node[rotate=150] {$\dots$};
                \draw (V23) -- ($(V23) + (150:1)$) node[rotate=150] {$\dots$};
                \draw (V23) -- ($(V23) + (270:1)$) node[rotate=270] {$\dots$};
                \draw (V31) -- ($(V31) + (150:1)$) node[rotate=150] {$\dots$};
                \draw (V31) -- ($(V31) + (270:1)$) node[rotate=270] {$\dots$};
                \draw (V32) -- ($(V32) + (270:1)$) node[rotate=270] {$\dots$};
                \draw (V32) -- ($(V32) + (30:1)$) node[rotate=30] {$\dots$};
            \end{scope}
        \end{tikzpicture}
    \end{center}
    \caption{%
        The Markoff tree~$\mathcal{G^\times}$.%
        \label{fig:markoff-tree}
    }
\end{figure}%

By the \textit{Markoff graph} we mean the graph~$\mathcal{G} = \mathcal{G}^\times \cup \{(0, 0, 0)\}$ where the trivial solution is an isolated node.

We may also consider the Markoff equation modulo a prime~$p > 2$,
\begin{align*}
    x^2 + y^2 + z^2 - xyz \equiv 0 \bmod p. \label{eq:markoff-mod-p}
\end{align*}

The solutions are denoted $\mathcal{M}_p$.
Define~$\mathbb{M}_p$ to be the integers~$x \in \mathbb{Z} / p\mathbb{Z}$ which appear in some triple~$(x, y, z) \in \mathcal{M}_p$, just like we did before.
Lemma~\ref{lem:mp} provides the following characterization:
\[ \mathbb{M}_p =
    \begin{cases}
        \mathbb{Z} / p\mathbb{Z} & -1 \text{\ is a quadratic nonresidue modulo } p,\\
        \mathbb{Z} / p\mathbb{Z} \setminus \{-2, 2\} & \text{otherwise}.
    \end{cases}
\]

\begin{lemma}\label{lem:mp}
    Let~$z \in \mathbb{Z} / p\mathbb{Z}$. 
    Then there are~$x, y \in \mathbb{Z} / p\mathbb{Z}$ such that
    \[ x^2 + y^2 + z^2 - xyz \equiv 0 \bmod p. \]
    if and only if either $z \neq 2$ or $-1$ is a quadratic residue modulo~$p$.
\end{lemma}
\begin{proof}
    The solutions~$(x, y, z) \in \mathcal{M}_p$ are given by the quadratic equation~\eqref{eq:quadratic}.
    If~$z = \pm 2$, then the discriminant of~\eqref{eq:quadratic} is~$-16$, which is in~$\mathbb{Z} / p\mathbb{Z}$ exactly when~$-1$ is a quadratic residue modulo~$p$.

    Now, fix~$z \neq \pm 2$.
    The map
    \[ y \mapsto y^2z^2 - 4(y^2 + z^2)\] 
    is 2-to-1 from~$\mathbb{Z} / p\mathbb{Z} \setminus \{0\}$ to~$\mathbb{Z} / p\mathbb{Z}$.
    There are~$(p + 1) / 2$ integers in the image of the map.
    But since there are only~$(p - 1) / 2$ quadratic non-residues, there must be a~$y$ where the discriminant of~\eqref{eq:quadratic} is a quadratic residue giving a~$x$ value for which~$(x, y, z) \in \mathcal{M}_p$.
\end{proof}

As before,~$\mathcal{M}_p$ is~$\Gamma$-fixed, and the only fixed point is~$(0, 0, 0)$.
We build a graph $\mathcal{G}_p^\times$ analogous to the Markoff tree with vertex set~$\mathcal{M}_p^\times$ and edges given by the Markoff Vieta involutions.
See Figure~\ref{fig:markoff-graph} for~$\mathcal{G}_5^\times$ as an example.

\begin{figure}
    \begin{center}
        \begin{tikzpicture}
            \def\d{1.25}
            \draw node (333) {$333$};

            \node (133) at ($(333) + (30:\d)$) {$133$};
            \node (130) at ($(133) + (-30:\d)$) {$130$};
            \node (120) at ($(130) + (30:\d)$) {$120$};
            \node (122) at ($(120) + (90:\d) + (-0.2,0)$) {$122$};
            \node (102) at ($(122) + (150:\d)$) {$102$};
            \node (103) at ($(102) + (210:\d)$) {$103$};
            \draw (333) -- (133);
			\draw (133) -- (130);
			\draw (130) -- (120);
			\draw (120) -- (122);
			\draw (122) -- (102);
			\draw (102) -- (103);
            \draw (103) -- (133);

            \node (430) at ($(130) + (300:\d)$) {$430$};
            \node (420) at ($(120) + (300:\d)$) {$420$};
			\draw (130) -- (430);
			\draw (120) -- (420);
            \draw (430) -- (420);

            \node (403) at ($(103) + (120:\d)$) {$403$};
            \node (402) at ($(102) + (120:\d)$) {$402$};
			\draw (103) -- (403);
			\draw (102) -- (402);
            \draw (403) -- (402);

            \draw ($(102) + (30:\d)$) node (423) {$423$};
            \draw ($(120) + (30:\d)$) node (432) {$432$};

            \node (322) at ($(122) + (210:\d)$) {$322$};
			\draw (122) -- (322);

            \draw (403) -- (423);
            \draw (430) -- (432);
            \draw (420) -- (423);
            \draw (432) -- (402);

            \node (331) at ($(333) + (150:\d)$) {$331$};
            \node (301) at ($(331) + (90:\d)$) {$301$};
            \node (201) at ($(301) + (150:\d)$) {$201$};
            \node (221) at ($(201) + (210:\d) + (0.2,0)$) {$221$};
            \node (021) at ($(221) + (270:\d)$) {$021$};
            \node (031) at ($(021) + (330:\d) + (-0.1,0)$) {$031$};
			\draw (333) -- (331);
			\draw (331) -- (301);
			\draw (301) -- (201);
			\draw (201) -- (221);
			\draw (221) -- (021);
			\draw (021) -- (031);
            \draw (031) -- (331);

            \node (304) at ($(301) + (420:\d)$) {$304$};
            \node (204) at ($(201) + (420:\d)$) {$204$};
			\draw (301) -- (304);
			\draw (201) -- (204);
            \draw (304) -- (204);

            \node (034) at ($(031) + (240:\d)$) {$034$};
            \node (024) at ($(021) + (240:\d)$) {$024$};
			\draw (031) -- (034);
			\draw (021) -- (024);
            \draw (034) -- (024);

            \draw ($(021) + (150:\d)$) node (234) {$234$};
            \draw ($(201) + (150:\d)$) node (324) {$324$};

            \node (223) at ($(221) + (330:\d)$) {$223$};
			\draw (221) -- (223);

            \draw (034) -- (234);
            \draw (304) -- (324);
            \draw (204) -- (234);
            \draw (324) -- (024);

            \node (313) at ($(333) + (270:\d)$) {$313$};
            \node (013) at ($(313) + (210:\d)$) {$013$};
            \node (012) at ($(013) + (270:\d)$) {$012$};
            \node (212) at ($(012) + (330:\d)$) {$212$};
            \node (210) at ($(212) + (390:\d)$) {$210$};
            \node (310) at ($(210) + (450:\d)$) {$310$};
			\draw (333) -- (313);
			\draw (313) -- (013);
			\draw (013) -- (012);
			\draw (012) -- (212);
			\draw (212) -- (210);
			\draw (210) -- (310);
            \draw (310) -- (313);

            \node (043) at ($(013) + (540:\d)$) {$043$};
            \node (042) at ($(012) + (540:\d)$) {$042$};
			\draw (013) -- (043);
			\draw (012) -- (042);
            \draw (043) -- (042);

            \node (340) at ($(310) + (360:\d)$) {$340$};
            \node (240) at ($(210) + (360:\d)$) {$240$};
			\draw (310) -- (340);
			\draw (210) -- (240);
            \draw (340) -- (240);

            \draw ($(210) + (270:\d)$) node (342) {$342$};
            \draw ($(012) + (270:\d)$) node (243) {$243$};

            \node (232) at ($(212) + (450:\d)$) {$232$};
			\draw (212) -- (232);

            \draw (340) -- (342);
            \draw (043) -- (243);
            \draw (042) -- (342);
            \draw (243) -- (240);

            \draw[dashed] (322) -- (324);
            \draw[dashed] (322) -- (342);
            \draw[dashed] (232) -- (432);
            \draw[dashed] (232) -- (234);
            \draw[dashed] (223) -- (423);
            \draw[dashed] (223) -- (243);
        \end{tikzpicture}
    \end{center}
    \caption{%
        The Markoff graph~$\mathcal{G}^\times_5$.%
        \label{fig:markoff-graph}
    }
\end{figure}
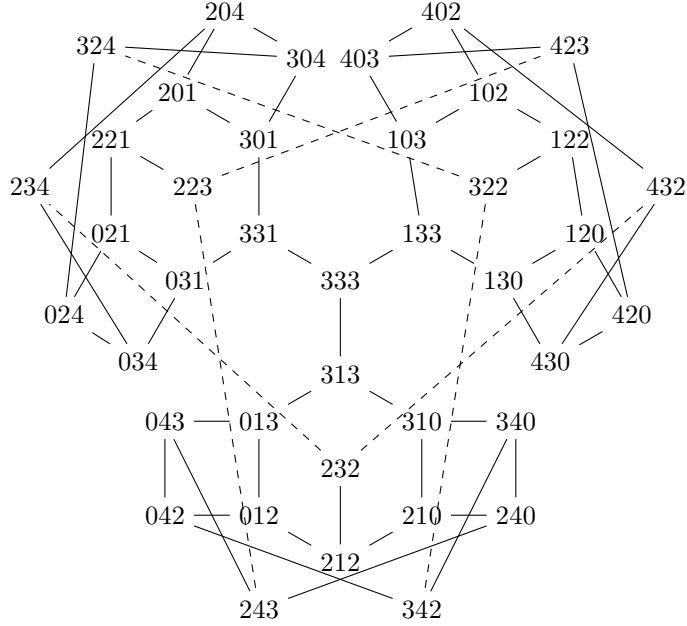%

The following conjecture was first made by~\citeauthor{strong-approx}~\cite{strong-approx}.

\begin{conjecture}\label{con:graphs-connected}
    For every prime~$p$, the Markoff graph~$\mathcal{G}^\times_p$ is connected and equivalent to~$\mathcal{G}^\times$ with triples congruent modulo~$p$ identified.
\end{conjecture}
\begin{conjecture}\label{con:triples-onto}
    For every prime~$p$, the canonical projection~$\mathcal{M} \to \mathcal{M}_p$ is onto.
\end{conjecture}

\begin{lemma}
    Conjectures~\ref{con:graphs-connected} and~\ref{con:triples-onto} are equivalent.
\end{lemma}
\begin{proof}
    If~$\mathcal{G}^\times_p$ is connected, then for any~$v \in \mathcal{M}_p$, there is a sequence~$V_{i_1},\dots V_{i_n}$ forming path from~$(3, 3, 3)$ to~$v$.
    Taking the vertices in this path modulo $p$,
    \[ (V_{i_1} \circ \cdots \circ V_{i_n})(3, 3, 3) \equiv v \bmod p. \]

    Conversely, assume~$\mathcal{M} \to \mathcal{M}_p$ is onto.
    For~$v \in \mathcal{M}_p$, let~$w \in \mathcal{M}$ such that~$w \equiv v \bmod p$.
    Since~$\mathcal{G}^\times$ is connected, there is a path~$S$ from~$(3, 3, 3)$ to~$w$.
    Now,~$S$ also forms a walk in~$\mathcal{G}^\times_p$, connecting~$(3, 3, 3)$ to~$v$.
    And since every vertex is connected to~$(3, 3, 3)$, they are all connected.
\end{proof}
Conjecture~\ref{con:triples-onto} is often called ``the strong approximation conjecture''.

\citeauthor{bgs} developed a three-tiered approach in~\cite{bgs} showing that strong approximation holds for all primes $p$ outside a zero density subset.
This result was strengthened by \citeauthor{chen} in~\cite{chen} to strong approximation holding for all but a finite set of primes.
\citeauthor{fuchs} proved an upper bound of approximately $3.448 \cdot 10^{392}$ on $p$ for which~$\mathcal{G}_p^\times$ may be disconnected~\cite{fuchs}.
Their method also proves that~$\mathcal{G}_p^\times$ is connected for many primes~beginning with~$p = 1,327,363$, but is inconclusive for most primes less than~$10^9$.

The strong approximation conjecture implies that solutions to the Markoff equation modulo a prime lift to Markoff triples over~$\mathbb{Z}$.
\citeauthor{bgs} use the lifting property to show that almost all Markoff numbers are highly composite~\cite[Theorem 18]{bgs}.
The strong approximation conjecture is the first step towards proving the much larger conjecture that~$\mathcal{G}_p^\times$ forms an expander family, first proposed in~\cite{bgs}.

In the course of studying the spectral gaps on Markoff graphs in~\cite{experiments}, \citeauthor{experiments} performed computations showing that~$\mathcal{G}_p^\times$ is connected for all primes~$p < 3,000$.
However, their method requires calculating the adjacency matrix of~$\mathcal{G}_p^\times$, which has on the order of~$O(p^4)$ entries, and is therefore infeasible for determining connectivity for larger values of~$p$.
Our work aims to fill in the gap between~$p = 3,001$ and~$p = 1,327,363$, and provide affirmative answers on the set of primes for which the results of~\cite{fuchs} are inconclusive.

In this paper, we outline an algorithmic realization of the process introduced by~\citeauthor{bgs}, along with an implementation of that algorithm in Rust.
Our algorithm runs in~$o(p^{1 + \epsilon})$ time for every~$\epsilon > 0$, and the real-world runtime is improved by incorporating the results of~\cite{fuchs}.
Our algorithm is suitable for arbitrarily large primes, limited only by hardware and real-world time constraints.
This is significantly better than a greedy approach with a flood fill-like algorithm with an~$O(p^2)$ runtime, there being~$p^2 + O(p)$ solutions to the Markoff equation modulo~$p$~\cite[Proposition 2.1]{experiments}.
We demonstrate our implementation by confirming the following result.

\begin{theorem}\label{thm:comp-results}
    The graph~$\mathcal{G}_p^\times$ is connected for all primes~$p$ less than one million.
\end{theorem}

In Section 2, we outline the arguments of~\cite{bgs} and~\cite{fuchs}, elaborating on the features we will utilize in our algorithm.
In Section 3, we construct the data structures and procedures at the core of our algorithm.
Finally, in Section 4, we present the algorithm in total, along with our computational results for all primes below one million, as well as auxillary data useful for short circuiting the computations and improving real world performance.
The Appendix includes implementation specific details regarding the Rust programming language and the Rust features we utilized most heavily: constant type parameterization and monomorphization.

\subsection*{Acknowledgements}
Many thanks to my advisor Elena Fuchs for her mentorship, encouragement, and suggestions throughout this paper.
I also thank Matt Litman for helpful conversations, especially with regards to Section 2.
Finally, thanks to Daniel Martin and Peter Sarnak for helpful feedback on an earlier draft.
This paper is based on work supported by the National Science Foundation under grant DMS-2154624.

\section{The structure of the Markoff graph modulo~$p$}

In this section, we review some of the structural properties of~$\mathcal{G}_p^\times$.
We begin our discussion with the orbits of the rotation maps, which will form useful walks along the graph.

Let~$\mathcal{D}(n)$ be the set of positive divisors of~$n$, and~$\mathcal{D}(p \pm 1) = \mathcal{D}(p - 1) \cup \mathcal{D}(p + 1)$.

We will consider the subgroup of~$\Gamma$ generated by the \textit{rotation maps}, which are compositions of a Markoff Vieta involution with a transposition.
If~$i \in \mathbb{Z}/3\mathbb{Z}$, then
\[ \rot_{i} = \tau_{i + 1, i + 2} \circ V_{i + 1} \]
is the rotation map fixing coordinate~$i$.
Writen explicitly,
\begin{align*}
    rot_1 &\colon (a, b, c) \mapsto (a, c, ac - b), \\
    rot_2 &\colon (a, b, c) \mapsto (ab - c, b, a), \\
    rot_3 &\colon (a, b, c) \mapsto (b, bc - a, c).
\end{align*}

Fix~$(a, b, c) \in \mathcal{M}_p$.
Denote the rotation map~$\rot_1$ restricted to acting on triples with fixed first coordinate~$a$ by
\begin{align*}
        \rot'_a \begin{pmatrix*} b \\ c \end{pmatrix*} &= \begin{pmatrix*} c \\ ac - b \end{pmatrix*} \\
                                                          &= \begin{pmatrix*} 0 & 1 \\ -1 & a \end{pmatrix*} \begin{pmatrix*} b \\ c \end{pmatrix*}.
\end{align*}

Let~$a \neq \pm 2$.
If~$a = \chi + \chi^{-1}$ for~$\chi \in \mathbb{F}_{p^2}$, then the matrix diagonalizes as
\begin{equation}
    \begin{pmatrix*} 0 & 1 \\ -1 & \chi + \chi^{-1} \end{pmatrix*} =
    \begin{pmatrix*}
        1 & 1 \\
        \chi & \chi^{-1}
    \end{pmatrix*}
    \begin{pmatrix*}
        \chi & 0 \\
        0 & \chi^{-1}
    \end{pmatrix*}
    \begin{pmatrix*}
        1 & 1 \\
        \chi & \chi^{-1}
    \end{pmatrix*}^{-1}. \label{eq:diagonalization}
\end{equation}

The size of the orbit of~$\rot_1$ acting on~$(a, b, c)$ is dependent on~$a$ only, and will be the same for any orbit with fixed first coordinate~$a$.
We therefore refer to the \textit{order of a Markoff number~$a$ modulo~$p$} as the size of any orbit with fixed first coordinate~$a$.
According to \eqref{eq:diagonalization}, this is equivalent to
\[ \ord_p(a) =  |\chi| \]
when $a \neq \pm 2$, where~$|\cdot|$ is the multiplicative order in~$\mathbb{F}_{p^2}^\times$.

The following lemma guarentees that the order of $\chi$ always divides $p \pm 1$, a fact we will use in our classification of Markoff triples.

\begin{lemma}\label{lem:order-divides}
    Let~$\chi \in \mathbb{F}_{p^2}^\times$.
    Then~$(\chi + \chi^{-1}) \in \mathbb{Z} / p\mathbb{Z}$ if and only if~$|\chi| \in \mathcal{D}(p \pm 1)$.
\end{lemma}
\begin{proof}
    If~$\chi \in \mathbb{F}_{p^2}^\times$ with~$|\chi| \in \mathcal{D}(p \pm 1)$, then
    \begin{align*}
        {\left(\chi + \frac{1}{\chi}\right)}^p &= \chi^p + \frac{1}{\chi^p} \\
                                             &= \chi^{\mp 1} + \frac{1}{\chi^{\mp 1}} \\
                                             &= \chi + \frac{1}{\chi},
    \end{align*}
    and therefore~$|\chi + \chi^{-1}| \in \mathcal{D}(p - 1)$, equivalently~$(\chi + \chi^{-1}) \in \mathbb{Z} / p\mathbb{Z}$.

    Since~$\mathbb{F}_{p^2}^\times$ is cyclic, there are~$2p$ values~$\chi \in \mathbb{F}_{p^2}^\times$ for which $|\chi|$ is divisible by~$p - 1$ or~$p + 1$.
    The map~$\phi : \chi + \chi^{-1}$ is 2-to-1 on~$\mathbb{F}_{p^2}^\times$, so the image of~$\phi$ on the set~$\{\chi : |\chi| \in \mathcal{D}(p \pm 1)\}$ is~$p$ values in~$\mathbb{Z} / p\mathbb{Z}$.
    But, that is all of $\mathbb{Z} / p\mathbb{Z}$, so there must be no more values of~$\chi \in \mathbb{F}_{p^2}^\times$ for which~$\chi + \chi^{-1} \in \mathbb{Z} / p\mathbb{Z}$.
\end{proof}

The diagonalization \eqref{eq:diagonalization} allows us to parameterize the Markoff triples in an orbit in terms of $\chi$.
Specifically, if~$(a, b, c) \in \mathcal{M}_p$ and~$a = \chi + \chi^{-1}$, then
\begin{align*}
    \langle \rot'_a\rangle(b, c) = \left\{ k\left(\alpha \chi^\ell + \beta \frac{1}{\chi^\ell}, \alpha \chi^{\ell + 1} + \beta \frac{1}{\chi^{\ell + 1}}\right) : \ell \in \mathbb{Z} \right\},
\end{align*}
where
\begin{align*}
    k = \left(\chi - \frac{1}{\chi}\right)^{-1},\; \alpha = \left(c - \frac{b}{\chi}\right),\text{\ and } \beta = \left(\chi b - c\right).
\end{align*}
A parameterization for arbitrary orbits is given in~\cite[Equation 3]{fuchs} as
\begin{align}
    \left\{ \left( \chi + \frac{1}{\chi}, \frac{\chi + \chi^{-1}}{\chi - \chi^{-1}}\left(r\chi^\ell + \frac{1}{r\chi^\ell}\right), \frac{\chi + \chi^{-1}}{\chi - \chi^{-1}}\left(r\chi^{\ell + 1} + \frac{1}{r\chi^{\ell + 1}}\right)\right) : \ell \in \mathbb{Z} \right\}, \label{eq:markoff-reparam}
\end{align}
where any values~$r, \chi \in \mathbb{F}_{p^2}^\times \setminus \{\pm 1\}$ gives a set of triples solving the Markoff equation over~$\mathbb{F}_{p^2}$ fixed by~$\rot_1$.
For fixed~$\chi \in \mathbb{F}_{p^2}^\times$, the orbits given by $r_1, r_2 \in \mathbb{F}_{p^2}^\times$ according to~\eqref{eq:markoff-reparam} will be the same exactly when~$r_1\langle \chi \rangle = r_2\langle \chi \rangle$.
We therefore seek, for fixed~$\chi \in \mathbb{F}_{p^2}^\times$, one representative $r$ for each coset of~$\langle \chi \rangle$ giving
\[ \frac{r + r^{-1}}{\chi - \chi^{-1}} \in \mathbb{Z} / p\mathbb{Z}, \]
so that our expression~\eqref{eq:markoff-reparam} gives solutions to the Markoff equation in~$\mathbb{Z} / p\mathbb{Z}$.
The solutions are characterized by Lemma~\ref{lem:correct-subgroup}.

\begin{lemma}\label{lem:correct-subgroup}
    Let~$\chi, r \in \mathbb{F}_{p^2}^\times$ and $\chi \neq \pm 1$.
    Let~$k \in \mathbb{Z} / p\mathbb{Z}$ be a quadratic nonresidue.
    \begin{enumerate}[label=(\alph*), ref=(\alph*)]
        \item \label{enum:hyperbolic}
            If~$\chi \in \mathbb{F}_p^\times$, then~$\chi - \chi^{-1} \in \mathbb{F}_p$.
        \item \label{enum:elliptic}
            If~$|\chi| \in \mathcal{D}(p + 1)$, then~$\chi - \chi^{-1} \in \sqrt{k} \mathbb{F}_p$.
        \item \label{enum:fixed-cosets}
            If~$|r| \in \mathcal{D}(2(p + 1)) \setminus \mathcal{D}(p + 1)$, then~$r + r^{-1} \in \sqrt{k}\mathbb{F}_p$.
    \end{enumerate}
\end{lemma}
\begin{proof}
    The proof of~\ref{enum:hyperbolic} is identical to the proof of Lemma~\ref{lem:order-divides} where~$|\chi| \in \mathcal{D}(p - 1)$.
    So, assume~$|\chi| \in \mathcal{D}(p + 1)$.
    Then,
    \begin{align*}
        \left(\chi - \frac{1}{\chi}\right)^p = \frac{1}{\chi} - \chi,
    \end{align*}
    therefore $|\chi - \chi^{-1}| \not\in \mathcal{D}(p - 1)$.
    However,
    \begin{align*}
        \left(\chi - \frac{1}{\chi}\right)^{2p} &= \left(\chi^2 + \frac{1}{\chi^2} - 2\right)^p \\
                                                            &= \left(\frac{1}{\chi^2} + \chi^2 - 2\right),
    \end{align*}
    so $(\chi - \chi^{-1})^2$ has multiplicative order dividing $p - 1$.
    Finally, let $|r| \in \mathcal{D}(2(p + 1))$.
    Then,
    \begin{align*}
        \left(r + \frac{1}{r}\right)^{2p} &= \left(r^2 + \frac{1}{r^2} + 2\right)^p \\
                                          &= \left(\frac{1}{r^2} + r^2 + 2\right),
    \end{align*}
    therefore $(r + r^{-1})^2$ has multiplicative order dividing $p - 1$.
\end{proof}

\begin{corollary}\label{cor:elliptic-magic}
    Let~$\chi, r \in \mathbb{F}_{p^2}^\times$ with~$|\chi| \in \mathcal{D}(p + 1)$,~$|r| \in \mathcal{D}(2(p + 1)) \setminus \mathcal{D}(p + 1)$, and~$\chi \neq \pm 1$.
    Then,~$\chi - \chi^{-1} = b_1\sqrt{k}$ and $r + r^{-1} = b_2\sqrt{k}$ for~$b_1, b_2, k \in \mathbb{Z} / p\mathbb{Z}$ and~$k$ a quadratic nonresidue.
    With these choices of~$\chi$ and~$r$, the orbit~\eqref{eq:markoff-reparam} contains triples with all coordinates in $\mathbb{Z} / p\mathbb{Z}$.
\end{corollary}

In the spirit of \cite{bgs}, we say that an element~$a \in \mathbb{Z}/p\mathbb{Z}$ is \textit{parabolic} if $a = \pm 2$.
Otherwise, it is \textit{hyperbolic} if~$\ord_p(a)$ divides~$p - 1$, or~\textit{elliptic} if~$\ord_p(a)$ divides~$p + 1$.

\begin{remark}\label{rmk:parabolic}
The orbits of triples with fixed parabolic coordinates have sizes~$\ord_p(a) = p$ or~$2p$.
If~$a$ is parabolic and~$a = \chi + \chi^{-1}$, then~$\chi = \pm 1$.
The matrix in \eqref{eq:diagonalization} is not invertible, and~$|\chi|$ is~$1$ or~$2$.
\end{remark}

\begin{definition}\label{def:triple-order}
    The \textit{order of the triple~$(a, b, c) \in \mathcal{M}_p$} is
    \[ \ord_p((a, b, c)) = \max(\ord_p(a), \ord_p(b), \ord_p(c)). \]
    If the order of a triple is~$p \pm 1$, then it is called \textit{a triple of maximal order}.
\end{definition}

Bourgain, Gamburd, and Sarnak showed in~\cite{bgs} that~$\mathcal{G}_p$ is connected for every prime~$p$ outside a zero density subset of the primes.
They did so by first observing that all triples of maximal order and parabolic triples belong to the same connected component $\mathcal{C}_p$.
(An explicit proof that triples with parabolic coordinates are connected to triples of maximal order is given in~\cite[Lemma 3.3]{bounding-lifts}.)
Then, they iteratively lower the bound on the orders of triples which must be connected to~$\mathcal{C}_p$.
In the first step, the \textit{endgame}, all triples of order at least~$p^{\delta + 1/2}$ are shown to be connected to~$\mathcal{C}_p$ for some~$\delta$ dependent on~$p$.
In the \textit{middle game}, they show that there is an~$\epsilon > 0$ dependent only on~$p$ for which~$\ord_p(a) > p^\epsilon$ implies that every orbit of~$\rot'_a$ contains a triple of order larger than~$\ord_p(a)$.
By successively walking from triple to triple of higher order, eventually a triple of order at least~$p^{\delta + 1/2}$ is reached.
Finally, in the \textit{opening}, they bound the size of~$|\mathcal{G}_p \setminus \mathcal{C}_p|$ from below, and show that for primes~$p$ with~$p \pm 1$ not having too many factors, there can not be enough unaccounted for triples to constitute a connected component disconnected from $\mathcal{C}_p$.
Specifically,~\citeauthor{bgs} proved in~\cite{bgs} a lower bound of $(\log p)^{1/3}$ on the size of any connected component.
This bound was improved by~\citeauthor{on-the-structure}~\cite{on-the-structure} to~$(\log p)^{7/9}$, and again improved by~\citeauthor{chen}~\cite{chen} to the size being divisible by~$p$.

The bounds on the endgame were made explicit in~\cite{fuchs}, where the following inequality was proved.

\begin{lemma}[{\cite[Proposition 6.1]{fuchs}}]\label{lem:endgame}
    A Markoff triple of order~$d \in \mathcal{D}(p \pm 1)$ is connected to $\mathcal{C}_p$ provided
    \begin{align}
        d > \frac{8\sqrt{p}(p \pm 1)\tau(p \pm 1)}{\phi(p \pm 1)}. \label{eq:endgame}
    \end{align}
\end{lemma}

We will call the right hand side of~\eqref{eq:endgame} for each choice of sign the \textit{endgame breakpoints} and denote them~$B_{(p,\pm)}$.

We now wish to bound from below the orders of triples which are part of the middle game.
We will consider an orbit of size~$t$, and give an upper bound on the number of triples in that orbit with order less than~$t$.
If the number of such triples is less than~$t$, then the orbit being considered must contain a triple of order larger than~$t$.
A suitable bound is given in~\cite{fuchs} and relies on an approximation of~\citeauthor{gcd-finite-fields}~\cite{gcd-finite-fields}.

\begin{lemma}[{\cite[Lemma 2.1]{fuchs}}]
        If~$\chi \in \mathbb{F}_{p^2}^\times$ has order~$t > 2$, then the number of congruence classes~$n$ modulo~$t$ for which
        \[ \ord_p\left(\frac{\chi + \chi^{-1}}{\chi - \chi^{-1}}\right){\left(sr^n + {(sr^n)}^{-1}\right)} \]
        divides~$d$ is at most
        \begin{align}
            \frac{3}{2}\max\left(\sqrt[3]{6td}, \frac{4td}{p}\right). \label{eq:cz-bounds}
        \end{align}
\end{lemma}

We now introduce the notion of maximal divisors, coined in~\cite{fuchs}.

\begin{definition}[{\cite[Definition 1.2]{fuchs}}]
    Let~$n$ be a positive integer, and let~$x \in \mathbb{R}$.
    A positive divisor~$d$ of~$n$ is \textit{maximal with to respect to~$x$} if~$d \leq x$ and there is no other positive divisor~$d'$ of~$n$ such that~$d' \leq x$ and~$d \mid d'$.
    The set of maximal divisors of~$n$ with respect to~$x$ is denoted~$\mathfrak{M}_x(n)$.
\end{definition}

Fix a divisor~$t$ of~$p \pm 1$.
If we sum~\eqref{eq:cz-bounds} over all maximal divisors of~$p - 1$ and~$p + 1$ with respect to~$t$, then we have an upper bound on the number of triples with order less than~$t$ in each orbit of order~$t$.
Explicitly, if
\begin{align}
        t > \sum_{d \in \mathfrak{M}_t(p \pm 1)} \frac{3}{2}\max\left(\sqrt[3]{6td}, \frac{4td}{p}\right), \label{eq:middle-game-inequality}
\end{align}
then every orbit of order~$t$ is guaranteed to have a triple of order larger than~$t$.
(In fact, this is still double counting elements of order dividing more than one maximal divisor; sharpening this bound is a direction for future work.)

Let~$L_p$ be the smallest~$d \in \mathcal{D}(p \pm 1)$ such that every~$d' \geq d$, with $d' \in \mathcal{D}(p \pm 1)$, satisfies~\eqref{eq:middle-game-inequality}.
If it exists, then we call~$L_p$ the \textit{middle game breakpoint}.

\begin{lemma}\label{lem:middle-game}
    Every triple of order at least~$L_p$ is connected to~$\mathcal{C}_p$.
\end{lemma}
\begin{proof}
    Let~$v_1 \in \mathcal{M}_p^\times$ with order~$d > L_p$.
    Then,~$v_1$ is connected to some triple~$v_2$ of order~$d' \gneq d$.
    This argument may be repeated until a triple of order~$p \pm 1$ is found, at which point we have reached~$\mathcal{C}_p$.
\end{proof}

Algorithm~\ref{alg:middlegame} shows a procedure for calculating~$L_p$.
\begin{algorithm}[h]
        \SetAlgoLined%
        \SetKwData{Sum}{sum}
        \Input{prime~$p$.}
        \KwResult{$L_p$, either a real number or \None.}

        $L_p \leftarrow$ \None\;
        \For{$t$ \In $\mathcal{D}(p \pm 1) \setminus \{2\}$ in ascending order}{%
            \Sum~$\leftarrow$~$0$\;
            \For{$d$ \In $\mathcal{D}$, $d < t$}{%
                \Sum~$\leftarrow$~$\Sum + \max(\sqrt[3]{6dt}, 4dt/p)$\;
            }
            \If{$t \geq \Sum$ {\bf and}~$L_p$ {\bf is} \None}{%
                $L_p \leftarrow t$\;
            }
            \ElseIf{$t < \Sum$}{%
                $L_p \leftarrow$ \None\;
            }
        }
        \Return $L_p$\;

        \caption{Finding the middle game breakpoint $L_p$.}%
        \label{alg:middlegame}
\end{algorithm}%

In~\cite{chen},~\citeauthor{chen} showed that the size of any connected component disconnected from~$\mathcal{C}_p$ is divisible by~$p$.
This result was used in~\cite{fuchs} to show the following.

\begin{lemma}[{\cite[Lemma 2.2]{fuchs}}]\label{lem:connected-component}
    If~$p > 3$, then~$|\mathcal{G}_p \setminus \mathcal{C}_p|$ is divisible by~$4p$.
\end{lemma}

Lemma~\ref{lem:endgame} and Corollary~\ref{lem:middle-game} give two bounds on the orders of triples in~$\mathcal{C}_p$.
Our goal is now to count the number of triples which we have not yet shown are connected to~$\mathcal{C}_p$.
If the number of triples thus counted is less than~$4p$, then~$\mathcal{G}_p$ is connected by Lemma~\ref{lem:connected-component}.

\begin{definition}\label{def:small-coord}
    Let~$\phi$ be a boolean function~$\mathbb{F}_p \to \mathbb{B}$.
    A coordinate~$a$ has \textit{small order with respect to~$\phi$} if~$\phi(a) = 0$.
    The set of coordinates with small order is denoted~$\mathcal{S}_\phi$.
\end{definition}

We will classify the coordinates with small order based on the boolean function
\begin{align*}
    \phi(a) =
    \begin{cases}
        1 & a \text{\ is hyperbolic and } \ord_p(a) < \min(L_p, B_{(p,-)}), \\
        1 & a \text{\ is elliptic and } \ord_p(a) < \min(L_p, B_{(p,+)}), \\
        0 & \text{otherwise},
    \end{cases}
\end{align*}
where~$L_p$ and $B_{(p,\pm)}$ are the middle game and endgame breakpoints, respectively.

\begin{definition}\label{def:bad-triples}
    Let~$\varphi$ be a boolean function~$\mathcal{M}_p \to \mathbb{B}$.
    A triple~$v \in \mathcal{M}_p$ is \textit{bad with respect to~$\varphi$} if~$\varphi(v) = 0$.
    The set of bad triples is denoted~$\mathcal{B}_\varphi$.
\end{definition}

For the remainder of the paper, we will elide the reference to~$\phi$ and~$\varphi$ when referring to the coordinates with small order~$\mathcal{S} = \mathcal{S}_\phi$ and bad triples~$\mathcal{B} = \mathcal{B}_\varphi$.

We now wish to construct a boolean function~$\varphi$ with two competing mandates.
First,~$\varphi$ must be as sensitive as possible; that is, we wish to minimize the number of triples~$v$ connected to $\mathcal{C}_p$ but for which~$\varphi(v) = 0$.
Second, we would like~$\varphi$ to be easy to compute, in the sense that a computer can quickly identify bad triples with respect to~$\varphi$.
To this end, we have two strategies:

\begin{strategy}\label{str:cartesian}
    For each~$(a, b) \in \mathcal{S} \times \mathcal{S}$, determine the zero, one, or two values of~$c$ for which~$(a, b, c) \in \mathcal{M}_p$ using~\eqref{eq:quadratic}.
    Any triple found this way for which~$a$,~$b$, and~$c$ all have small order is a bad triple.
\end{strategy}

\begin{strategy}\label{str:cosets}
    For each~$a \in \mathcal{S}$ and~$a = \chi + \chi^{-1}$, and for each coset of~$\langle \chi \rangle$, pick a representative~$r \in \mathbb{F}_{p^2}^\times / \langle \chi \rangle$ and calculate the entries in the orbit of $\rot_a'$ given by~\eqref{eq:markoff-reparam}.
    Note that $r \in \mathbb{Z} / p\mathbb{Z}$ if $a$ is hyperbolic, and $r \in \mathbb{F}_{p^2}^\times$ with order dividing $2(p + 1)$ if $a$ is elliptic; see Corollary~\ref{cor:elliptic-magic}.
    Check a set number of triples in the orbit\footnotemark; if they are all of small order, then every triple in the orbit is bad.
    Otherwise, none of the triples in the orbit are bad.
\end{strategy}
\footnotetext{
    See Figure~\ref{fig:coset-checks} for the maximum number of triples checked per orbit.
    We cap the number of triples checked to prevent our program hanging on a prime for which an exceptional number of checks must be performed.
    In practice, we found no such exceptional prime out of all primes less than one million, or among those randomly sampled less than 110,000,000.
}

For every~$a \in \mathbb{F}_p$, we run one of the two above strategies to identify all bad triples with first coordinate~$a$.
We choose the strategy requiring fewer checks, i.e., we choose Strategy~\ref{str:cartesian} when~$|\mathcal{S} \times \mathcal{S}| < (p \pm 1) / |\chi|$, and Strategy~\ref{str:cosets} otherwise.

\begin{remark}
    Strategy~\ref{str:cosets} is more sensitive, since a triple may have all coordinates with small order and yet still be ``good''.
    It is possible that, for some prime~$p$, choosing to only run Strategy~\ref{str:cosets} could affirmatively show that~$\mathcal{G}_p$ is connected while sometimes running Strategy~\ref{str:cartesian} is inconclusive.

    Strategy~\ref{str:cosets} has a second advantage; it will only ever test Markoff triples.
    Strategy~\ref{str:cartesian}, on the other hand, must filter through~$(a, b)$ pairs for which no value of~$c$ gives a Markoff triple.
    The runtime cost of performing this filter is similar to the weakening of the estimations in~\cite{fuchs} caused by assuming each pair~$(a, b)$ has two values of~$c$ making~$(a, b, c) \in \mathcal{M}_p$.
    Nevertheless, when the middle game breakpoint or~$|\chi|$ is small, Strategy~\ref{str:cartesian} may still be faster.
\end{remark}

\section{Computational Machinery}

In this section, we describe the data structures and procedures needed to implement the strategies described in Section~2.
We begin with the \textit{factor trie}, encoding a canonical poset on~$\mathcal{D}(p \pm 1)$.
This poset will allow us to recursively generate only those elements~$a \in \mathbb{Z} / p\mathbb{Z}$ satisfying $\ord_p(a)$ less than some upper bound (or~$a$ being parabolic), without having to filter the entirety of~$\mathbb{Z} / p\mathbb{Z}$ by order.

For each~$n = p_1^{t_1} \cdots p_n^{t_n}$ with primes in increasing order we associate a word
\begin{align}
    w(n) = \underbrace{p_1 \ldots p_1}_{t_1} \ldots \underbrace{p_n \ldots p_n}_{t_n}. \label{eq:word}
\end{align}

Let~$\mathcal{D}(n)$ be the set of positive divisors of~$n$.
Define a partial ordering~$\prec$ on~$\mathcal{D}$ as~$\ell \prec m$ if the word~$w(\ell)$ is a prefix for the word~$w(m)$.
If~$\ell \prec m$ and there is no~$n$ such that~$\ell \prec n \prec m$, then we write~$\ell \precdot m$.

Define\footnotemark\ the factor trie~$T_n = (V, E)$ as the graph with vertex set~$V = \{w(d) : d \in \mathcal{D}(n)\}$ and a directed edge~$(\ell, m) \in E$ if~$\ell \precdot m$.
\footnotetext{%
    Some authors use this as the definition for a tree; others define a tree as a connected acylcic graph.
    These definitions are equivalent up to a relabeling of the vertex set.
    We use the term \textit{trie} to emphasize that the vertices are \textit{per se} their associated words.
}%
An example factor trie for~$n = 60$ is shown in Figure~\ref{fig:trie-example}.

\begin{figure}[h]
\begin{center}
        \begin{tikzpicture}
        [ level distance=35
        , level 1/.style={sibling distance=90}
        , level 2/.style={sibling distance=50}
        , level 3/.style={sibling distance=30}
        ]
                \node {$1$}
                    child {node {$2$} edge from parent[->]
                        child {node {$2^2$}
                            child {node {$2^2 \cdot 3$}
                                child {node {$2^2 \cdot 3 \cdot 5$}}
                            }
                            child {node {$2^2 \cdot 5$}}
                        }
                        child {node {$2 \cdot 3$}
                            child {node {$2 \cdot 3 \cdot 5$}}
                        }
                        child {node {$2 \cdot 5$}}
                    }
                    child {node {$3$} edge from parent[->]
                        child {node {$3 \cdot 5$}}
                    }
                    child {node {$5$} edge from parent[->]};
        \end{tikzpicture}
        \caption{The factor trie for~$n = 60$.}%
        \label{fig:trie-example}
\end{center}
\end{figure}
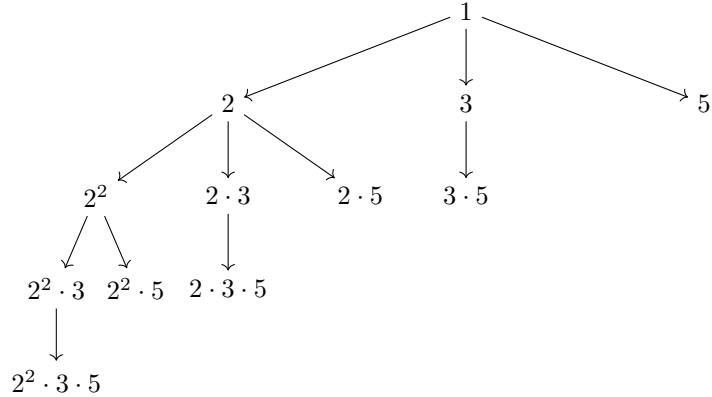%

Algorithm~\ref{alg:trie-construction} shows a recursive procedure for constructing factor tries.

\begin{algorithm}[h]
    \SetAlgoLined%
    \SetKwData{Bp}{bp}
    \SetKwData{Divisors}{divisors}
    \SetKwData{T}{t}
    \SetKwData{D}{d}
    \SetKwData{Sum}{sum}
    \SetKwFunction{NewTrie}{new\_trie}
    \Input{%
        $n = p_1^{t_1} \cdots p_n^{t_n}$, \\
        integers $d_1, \ldots, d_m$.
    }
    \KwResult{Factor trie $T$ for $p_1^{d_1} \cdots p_m^{d_m}$.}
    \Base{\NewTrie{$0$}.}

    \Fn{\NewTrie{$n$, $d_1,\ldots,d_m$}}{%
        $T$ $\leftarrow$ new node\;
        \For{$i$ \In $m$ \KwTo $n$}{%
            \If{$d_i < t_i$}{%
                $i$th child of $T$ $\leftarrow$ \NewTrie{$n$, $d_1,\ldots,d_{i - 1}, d_i + 1$}\;
            }
        }
        \Return $T$\;
    }

    \caption{Constructing a factor trie.}%
    \label{alg:trie-construction}
\end{algorithm}%

We have two tries for each prime $p$, given by~$T_{p + 1}$ and~$T_{p - 1}$.
Associate a group to each node of these tries via the functor~$\mathcal{F} \colon T_{p + 1} \amalg T_{p - 1} \to \mathbf{Ab}$ given by~$\mathcal{F} = \mathcal{F}_{p - 1} \amalg \mathcal{F}_{p + 1}$, where
\[ \mathcal{F}_{p \pm 1} \colon w(d) \mapsto \mathbb{Z} / d\mathbb{Z} \]
and~$\mathcal{F}_{p \pm 1}$ takes directed edges in the graph to inclusion maps.

The tries~$T_{p \pm 1}$ also form the call graph for our recursive procedure generating the elements of~$\mathbb{Z} / p\mathbb{Z}$.
If~$d_1 \prec d_2$, then~$|\mathcal{F}(w(d_1))| < |\mathcal{F}(w(d_2))|$; furthermore, if~$d_1 \precdot d_2$, then there is no intermediate subgroup between~$\mathcal{F}(w(d_1))$ and~$\mathcal{F}(w(d_2))$.
Therefore, our~$\chi$ values are generated in increasing multiplicative order moving down the call graph induced by~$T$, and we can stop the procedure when we have reached the desired bound.

We could represent elements of the group~$\mathcal{F}_{p \pm 1}(w(d))$ as integers modulo~$d$.
However, it will be more convenient to represent them instead as integer arrays based on the unique decomposition of these groups into a direct sum of prime power cyclic groups. 
The multiplicative group operations can then be performed by the CPU using only componentwise addition and subtraction, which are much faster than direct multiplication and reciprocal taking.
One drawback is that addition and subtraction are not easily represented, so a conversion back to~$\mathbb{Z} / p\mathbb{Z}$ is necessary before any ring operations.
We now describe this representation.

Let~$p - 1 = p_1^{t_1} \cdots p_n^{t_n}$ and~$p + 1 = q_1^{s_1} \cdots q_m^{s_m}$, and assume $p_1 = q_1 = 2$.

Fix, once and for all, a~$\mathbb{Z}$-basis of~$\mathbb{F}_p^\times$ of the form~$\{g_i\}_{i=1}^n$, where~$|g_i| = p_i^{t_i}$.
Making this choice is a special case of~\cite[Algorithm 6.1]{generators}, which is our preferred method.
This choice induces a group isomorphism
\[ \iota_{p - 1} : \mathbb{F}_p^\times \xrightarrow{\sim} \bigoplus_{i = 1}^n \mathbb{Z} / p_i^{t_i} \mathbb{Z} \]
given by
\begin{align}
    \prod_{i=1}^n g_i^{r_i} \mapsto (r_1,\ldots,r_n). \label{eq:generators}
\end{align}

An identical construction exists for the elliptic coordinates by replacing~$\mathbb{F}_p^\times$ with the subgroup~$E \subset \mathbb{F}_{p^2}^\times$ with order $p + 1$.
Taken together, we have the commutative diagram
\begin{center}
    \begin{tikzcd}
        & \mathbb{Z}/p\mathbb{Z}
        & \\
          \mathbb{F}_p^\times \arrow[ur, "\psi_{p - 1}" above] \arrow[r, "\iota_{p - 1}" below, hook]
        & \left(\bigoplus_{i = 1}^n \mathbb{Z} / p_i^{t_i} \mathbb{Z}\right) \amalg \left(\bigoplus_{i = 1}^m \mathbb{Z} / q_i^{s_i} \mathbb{Z}\right) \arrow[u, dashed]
        & E \arrow[ul, "\psi_{p + 1}" above] \arrow[l, "\iota_{p + 1}" below, hook']
        \end{tikzcd}
\end{center}
where~$\psi_{p \pm 1}(\chi) = \chi + \chi^{-1}$.
The subgroups of~$\mathbb{F}_p^\times$ and~$E$ correspond to the nodes of~$T_{p - 1}$ and~$T_{p + 1}$, respectively.
For~$d = p_1^{d_1} \cdots p_n^{d_n}$, we associate to the node~$w(d)$ the subgroup
\[ \mathcal{F}_{p - 1}(w(d)) = \mathbb{Z} / d\mathbb{Z} \cong \bigoplus_{i = 1}^n \mathbb{Z} / p_i^{d_i}\mathbb{Z} \]
containing integer arrays given by~\eqref{eq:generators}, and similarly for the divisors of~$p + 1$.

\begin{remark}\label{rmk:two-to-one}
    Since~$\psi_{p \pm 1}(\chi) = \psi_{p \pm 1}(\chi^{-1})$, the maps~$\psi_{p - 1}$ and~$\psi_{p + 1}$ are 2-to-1 when restricted to~$\mathbb{F}_p^\times \setminus \{\pm 1\}$ and~$E \setminus \{\pm 1\}$, respectively.
    Since~$\psi_{p - 1}(\pm 1) = \psi_{p + 1}(\pm 1)$, the induced (dotted) map in the commutative diagram is~2-to-1.
\end{remark}

We now describe the recursive procedure for generating a value~$a \in \mathcal{S}$ (see Definition~\ref{def:small-coord}) with order~$\ord_p(a) = d$.
We begin with the integer array~$(0, 0, \ldots, 0)$ associated to~$w(1)$, and propagate this array down the factor trie recursively, with each node~$w(d)$ yielding\footnotemark\ all the coordinates~$a \in \mathbb{Z} / p\mathbb{Z}$ with~$\ord_p(a) = d$, mapping the integer array to~$a$ via the inverse map of~\eqref{eq:generators}.
\footnotetext{A \texttt{yield} statement is similar to a \texttt{return} statement in that it passes data back to the caller of the function; it differs in that execution continues in the callee, allowing for more than one value to be yielded.}

\begin{remark}
    Parabolic triples are connected to~$\mathcal{C}_p$ by definition, so we do not want to generate~$\chi$ values for which~$|\chi| = 1$ or~$2$ (see Remark~\ref{rmk:parabolic}).
\end{remark}

In light of Remark~\ref{rmk:two-to-one}, we would like to yield only one representative from each~$\{\chi, \chi^{-1}\}$ pair.
The following definition and lemma provides a characterization for canonizing this choice.

\begin{definition}\label{def:lower-half}
    Let~$p \pm 1 = p_1^{t_1} \cdots p_n^{t_n}$ with $p_1 = 2$, and let~$(r_1, \ldots, r_n) \in \oplus_{i=1}^n \mathbb{Z}/p_i^{t_i}\mathbb{Z}$ be an integer array according to~\eqref{eq:generators} not representing a~$\chi \in \mathbb{F}_{p^2}^\times$ with order~$1$ or~$2$.
    Let
    \[ k = \min\{j : r_j \neq 0 \text{\ and } (p_j, r_j) \neq (2, 2^{t_1 - 1})\}. \]
    The integer array is \textit{in the lower half} if~$r_k \leq p_k^{t_k} / 2$.
\end{definition}

\begin{lemma}\label{lem:one-per-half}
    Let~$\chi \in \mathbb{F}_{p^2}^\times$ such that~$|\chi| \not\in \{1, 2\}$ and~$|\chi| \in \mathcal{D}(p \pm 1)$.
    Then, exactly one of~$\iota_{p \pm 1}(\chi)$ or~$\iota_{p \pm 1}(\chi^{-1})$ is in the lower half.
\end{lemma}
\begin{proof}
    Consider a hyperbolic~$a \in \mathbb{Z} / p\mathbb{Z}$ with~$a = \chi + \chi^{-1}$ and~$\chi \in \mathbb{F}_p^\times$.
    (The elliptic case is handled identically.)
    Let~$\chi = g_1^{r_1} \cdots g_n^{r_n}$.
    Since~$\iota_{p - 1}(\chi^{-1}) = (a_1, \ldots, a_n)$ where
    \[ a_i = (p_i^{t_i} - r_i) \bmod p_i^{t_i}, \]
    if~$|\chi| \not\in \{1, 2\}$, then at least one of~$r_k$ and~$p_k^{t_k} - r_k$ is less than or equal to~$p_k^{t_k} / 2$.
    Therefore, at least one of~$\iota_{p - 1}(\chi)$ and~$\iota_{p - 1}(\chi^{-1})$ is in the lower half.

    Now assume that both~$\chi$ and~$\chi^{-1}$ are in the lower half.
    For each~$j$ such that~$p_j \neq 2$, it must be that~$r_j = 0$.
    But then our array is either~$(0,\ldots,0)$ or~$(2^{t_1 - 1},\ldots,0)$, rearranging if necessary so~$p_1 = 2$, and these arrays correspond to~$|\chi| = 1$ and~$|\chi| = 2$, respectively.
\end{proof}

We will only yield integer arrays in the lower half, thus generating each desired element of~$\mathbb{Z} / p\mathbb{Z}$ exactly once.
We do this by setting upper bounds on the entries in the arrays yielded by node~$w(d)$ for~$d = p_1^{d_1} \cdots p_m^{d_m}$ according to
\begin{align}\label{eq:upper-bounds}
    r_j \leq \ell_{d,j} =
    \begin{cases}
        p_j^{t_j} / 2 & j = \min\{j : p_j^{d_j} \neq 2\}, \\
        p_j^{t_j} & \text{otherwise}.
    \end{cases}
\end{align}

We now present our algorithm for yielding Markoff coordinates~$a \in \mathcal{S}$ in Algorithm~\ref{alg:propagate}.
Each value of~$\chi$ yielded corresponds to a Markoff coordinate~$\chi + \chi^{-1}$.

\begin{algorithm}[h]
    \SetAlgoLined%
    \SetKwFunction{Propagate}{propagate}
    \Input{%
        Factor trie $T_{p \pm 1}$ with $p \pm 1 = p_1^{t_1} \cdots p_n^{t_n}$, \\
        Limits $\ell_1, \ldots, \ell_n$ as in~\eqref{eq:upper-bounds}, \\
        Node $w(d)$ with $d = p_1^{d_1} \cdots p_m^{d_m}$, \\
        Array of integers $r_1,\ldots,r_m$. \\
    }
    \Base{%
        \Propagate{$w(1)$, $0,\ldots,0$}
    }
    \KwResult{Stream of $\chi$ values.}

    \Fn(){\Propagate{$w(d)$, $r_1, \ldots, r_m$}}{%
        \For{$i$ \In~$0$ \KwTo\ $p_m$}{%
            $r'_m \leftarrow r_m + p_m^{t_m - d_m}$\;
            \If{$r'_m > \ell_{d,m}$}{%
                \Break\;
            }
            \If{$d_m < t_m$}{%
                \Propagate{$w(p_1^{d_1} \cdots p_m^{d_m + 1})$, $r_1,\ldots,r'_m$}\;
            }
            \If{$i = 0$}{%
                \tcc{Yielding or propagating when $i = 0$ would not respect the trie structure, i.e., there would be multiple call paths to yield the same value.}
                \Continue\;
            }
            \If{$2 < d < \min(L_p, B_{(p,\pm)})$}{%
                \Yield $g_1^{r_1} \cdots g_{m - 1}^{r_{m - 1}} \cdot g_m^{r'_m}$\;
            }
            \For{$j$ \In $m + 1$ \KwTo $n$}{%
                \If{$d_j < t_j$ \And $w(d \cdot p_j)$ or there is a $d' \succ d$ such that $d' < \min(L_p, B_{(p,\pm)})$}{%
                    \Propagate{$w(d \cdot p_j)$, $r_1,\ldots,r'_m$}\;
                }
            }
        }
    }
    \caption{%
        Generating $\chi$ corresponding to desired Markoff coordinates $\chi + \chi^{-1} \in \mathcal{S}$.%
        \label{alg:propagate}
    }
\end{algorithm}%

Besides generating all~$a \in \mathcal{S}$, we can use the factor trie to generate cosets of~$\langle \chi \rangle$ for $\chi \in \mathbb{F}_{p^2}^\times$, as in Strategy~\ref{str:cosets}.
For this we need one representative for each coset of~$\langle\chi\rangle$.
That is, for~$|\chi| \in \mathcal{D}(p \pm 1)$, we need one value of order~$d$ for each divisor~$d$ of~$(p \pm 1) / |\chi|$.
Let~$|\chi| = p_1^{s_1}\cdots p_n^{s_n}$.
We can repurpose Algorithm~\ref{alg:propagate} by instead interpreting the yielded values as~$r$ in~\eqref{eq:markoff-reparam}, and replacing the upper bounds in~\eqref{eq:upper-bounds} with
\begin{align}\label{eq:quotient-limit}
    \ell_{d,j} =
    \begin{cases}
        0 & d_j \leq s_j, \\
        p_j^{t_j} / 2 & d_j > s_j \text{ and } j = \min\{j : p_j^{d_j} \neq 2\}, \\
        p_j^{t_j} & \text{otherwise}.
    \end{cases}
\end{align}

\section{Our Algorithm \& Results}

In this section, we combine the process described in Section 2 with the machinery constructed in Section 3 to produce a comprehensive algorithm for proving that~$\mathcal{G}^\times_p$ is connected.
The code described in this section can be found at \url{https://github.com/colbyaustinbrown/libbgs}, with executable code found in the \texttt{examples} directory.
We used the Rust programming language, version~1.74, for our \texttt{libbgs} library defining the data structures described in Section 3.
Rust is a strongly, statically typed imperative language prioritizing speed, memory safety, and concurrency.
We now present our algorithm for testing whether $\mathcal{G}_p^\times$ is connected.

\FloatBarrier
\begin{algorithm-thm}[Testing~$\mathcal{G}_p$ for connectivity]\label{alg:process}
    \mbox{}
    \begin{enumerate}[ref=\arabic*]
        \item\label{step:make-tries}
            Calculate the prime factorizations for~$p \pm 1$, and construct the factor tries~$T_{p \pm 1}$.
            Also, fix~$\mathbb{Z}$-bases of~$\mathbb{F}_p^\times$ and~$E$ for~\eqref{eq:generators}, and also a value~$m$ with order~$2^t || p^2 - 1$ (see Corollary~\ref{cor:elliptic-magic}) using~\cite[Algorithm 6.1]{generators}.
        \item\label{step:breakpoints}
            Use~\eqref{eq:endgame} and Algorithm~\ref{alg:middlegame} to determine the endgame and middle game breakpoints.
        \item\label{step:generate-as}
            Generate all coordinates with small order~$a \in \mathbb{F}_p$ according to Definition~\ref{def:small-coord} via Algorithm~\ref{alg:propagate}, and proceed to step~\ref{step:main-loop} for each~$a$.
        \item\label{step:main-loop}
            If~$(p \pm 1) / \ord_p(a) < |\mathcal{S} \times \mathcal{S}|$, go to step~\ref{step:cosets}; otherwise, go to step~\ref{step:small-orders}.
            \begin{enumerate}[ref=\arabic{enumi}(\alph*)]
                \item\label{step:small-orders}
                    Perform Strategy~\ref{str:cartesian}:
                    Use Algorithm~\ref{alg:propagate} to generate all~$b \in \mathbb{F}_p$ with small orders.
                    For each~$b$ value, calculate the~$0$, $1$, or $2$ values of~$c$ for which~$(a, b, c) \in \mathcal{G}_p$.
                    Any~$c$ value with small order gives a bad triple~$(a, b, c)$.
                \item\label{step:cosets}
                    Perform Strategy~\ref{str:cosets}:
                    Generate all~$r$ values given by Algorithm~\ref{alg:propagate} modified with~\eqref{eq:quotient-limit}, and use it to calculate all~$b$ values from~\eqref{eq:markoff-reparam}.
                    (If $a$ is elliptic, the values $r$ must be multiplied by the value $m$ chosen in Step~\ref{step:make-tries} before being plugged in to \eqref{eq:markoff-reparam}, see Corollary~\ref{cor:elliptic-magic}.)
                    If every~$b$ value has order smaller than~$L$, then every triple in the orbit is bad.
            \end{enumerate}
        \item\label{step:total}
            Count the number of bad triples found.
            If the result is less than~$4p$, then~$\mathcal{G}_p$ is connected.
            Otherwise, the algorithm is inconclusive.
    \end{enumerate}
\end{algorithm-thm}%

Using Algorithm~\ref{alg:process}, we confirmed that~$\mathcal{G}_p$ is connected for all~$p < 1,000,000$, confirming Theorem~\ref{thm:comp-results}.
We also ran our algorithm for a random sample of $1,000$ primes~$p < 110,000,000$, and~$\mathcal{G}_p^\times$ was connected for each of these primes, too.
We did not find any~$\mathcal{G}_p$ with at least~$p$ bad triples, with two exceptions: $p = 7,558,541$ and $p = 96,840,901$, for which the number of bad triples found was $9,716,411$ and $103,370,751$, respectively.
The criteria for connectivity of~\cite{chen} was nearly always sufficient in our tests, although we did rely on the improved bounds of~\cite{fuchs} that any $|\mathcal{G}_p^\times \setminus \mathcal{C}_p| \geq 4p$.
The number of bad triples per prime is shown in Figure~\ref{fig:bad-triples}.
Table~\ref{tbl:most-smalls} shows the prime factorizations of $p \pm 1$ and bad triple counts for a sample of primes less than one million.
Generally, the number of bad triples is inversely correlated with the number of prime divisors (with multiplicity) of $p \pm 1$.

\begin{figure}[tbph]
\begin{center}
    \input{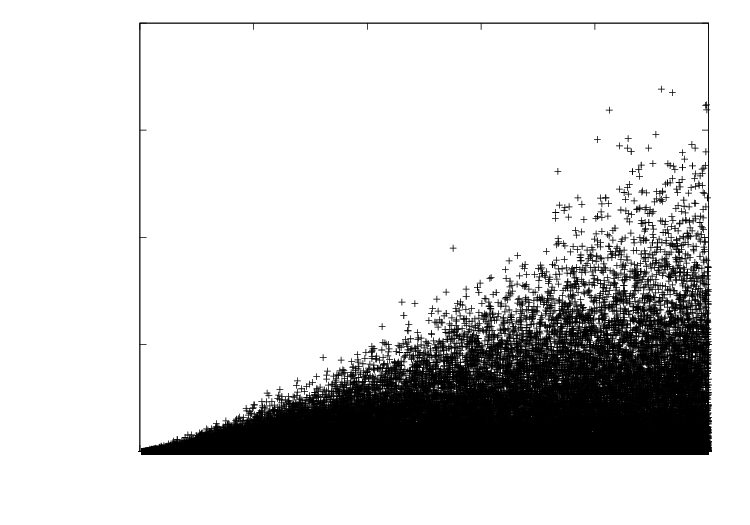}

    \input{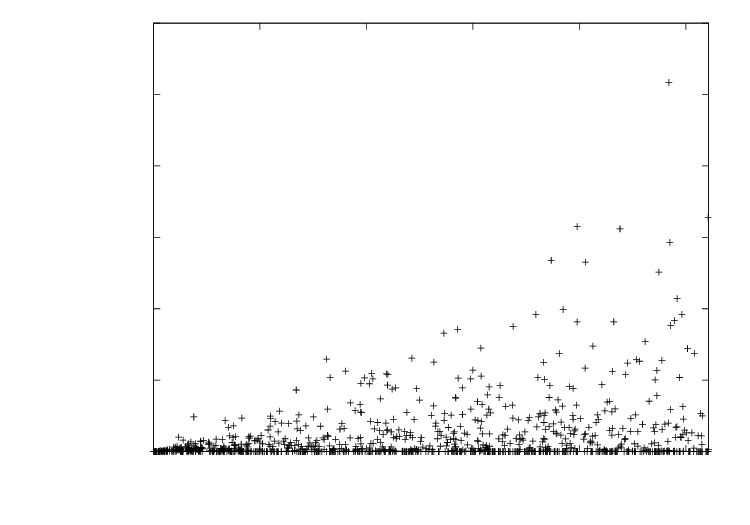}
    \caption{Total number of bad triples for all primes less than $1,000,000$ (top) and random primes less than $110,000,000$ (bottom).}
    \label{fig:bad-triples}
\end{center}
\end{figure}

\begin{table}[tph]
\begin{center}
    \caption{For $p < 1,000,000$, the $10$ primes with the largest ratio $|\mathcal{B}| / p$, and the~$5$ largest primes with~$\mathcal{B} = \emptyset$.}%
    \label{tbl:most-smalls}
    \begin{tabular}{ccccc}
        \toprule
        & \multicolumn{2}{c}{Hyperbolic} & \multicolumn{2}{c}{Elliptic} \\
        \cmidrule(lr){2-3}\cmidrule(lr){4-5}
        $p$ &~$p - 1$ & Bad Triples &~$p + 1$ & Bad Triples \\
        \midrule
        825,287	&~$2 \cdot 7 \cdot 11 \cdot 23 \cdot 233$ &	277,287 &~$2^3 \cdot 3 \cdot 137 \cdot 251$ &	320,209	    \\
        916,879	&~$2 \cdot 3 \cdot 17 \cdot 89 \cdot 101$ &	251,391	&~$2^4 \cdot 5 \cdot 73 \cdot 157$ &	425,410 \\
        804,203	&~$2 \cdot 7 \cdot 17 \cdot 31 \cdot 109$ &	295,979	&~$2^2 \cdot 3^2 \cdot 89 \cdot 251$ &	286,714	\\
        936,259	&~$2 \cdot 3 \cdot 17 \cdot 67 \cdot 137$ &	307,155	&~$2^2 \cdot 5 \cdot 13^2 \cdot 277$ &	362,722 \\
        734,803	&~$2 \cdot 3 \cdot 29 \cdot 41 \cdot 103$ &	171,268	&~$2^2 \cdot 7^2 \cdot 23 \cdot 163$ &	351,854	    \\
        550,811	&~$2 \cdot 5 \cdot 13 \cdot 19 \cdot 223$ &	211,593	&~$2^2 \cdot 3 \cdot 197 \cdot 233$ &	168,181 \\
        858,701	&~$2^2 \cdot 5^2 \cdot 31 \cdot 277$ &	279,547	&~$2 \cdot 3 \cdot 13 \cdot 101 \cdot 109$ & 304,704 \\
        843,229	&~$2^2 \cdot 3^2 \cdot 59 \cdot 397$ &	320,154	&~$2 \cdot 5 \cdot 37 \cdot 43 \cdot 53$ &	250,655	\\
        \midrule
        995,677 &~$2^2 \cdot 3 \cdot 11 \cdot 19 \cdot 397$ & 0 &~$2 \cdot 497,839$ & 0 \\
        995,987 &~$2 \cdot 497,993$ & 0 &~$2^2 \cdot 3 \cdot 7 \cdot 71 \cdot 167$ & 0 \\
        996,783 &~$2^2 \cdot 3 \cdot 5 \cdot 37 \cdot 449$ & 0 & $2 \cdot 498,391$ & 0 \\
        997,583 &~$2 \cdot 498,791$ & 0 &~$2^4 \cdot 3 \cdot 7 \cdot 7,969$ & 0 \\
        999,958 &~$2 \cdot 499,979$ & 0 &~$2^3 \cdot 3 \cdot 5 \cdot 13 \cdot 641$ & 0 \\
        \bottomrule
    \end{tabular}
\end{center}
\end{table}%

In the worst case, steps~\ref{step:make-tries} and~\ref{step:breakpoints} require on the order of~$|\mathcal{D}(p \pm 1)| + |\mathcal{D}(p + 1)|$ computations.
Step~\ref{step:main-loop} executes at most~$p$ times, and each iteration is bounded above by the runtime of step~\ref{step:cosets}, which checks at most~$\mathcal{D}(p \pm 1)$ orbits.
The check on each orbit terminates after finding any triple not in $\mathcal{B}$; the largest number of such checks per prime is shown in Figure~\ref{fig:coset-checks}.
We capped the number of checks at 60 to ensure the runtime stays on the order of~$|\mathcal{D}(\pm 1)|$.
Step~\ref{step:cosets} was terminated after~60 checks for only one prime, $p = 119,659$, and the algorithm still affirmitely determined $\mathcal{G}_p^\times$ is connected.
Notably, the number of checks required does not have any discernable growth over the much larger range $p < 110,000,000$.
We use the following result.
\begin{theorem}[{\cite[theorem 315]{theory-of-numbers}}]
   For every~$\epsilon > 0$, the divisor function~$\mathcal{D}(n) \in o(n^\epsilon)$.
\end{theorem}
Therefore, our algorithm runs in~$o(p^{1 + \epsilon})$ time, in the worst case.
The runtimes for all the primes we tested are shown in Figure~\ref{fig:runtimes}.
As emphasized in the figures, the almost linear bound on the runtime is a worst-case upper bound, and our algorithm performs well below this upper limit for a large number of primes.

\begin{figure}[tbph]
\begin{center}
    \input{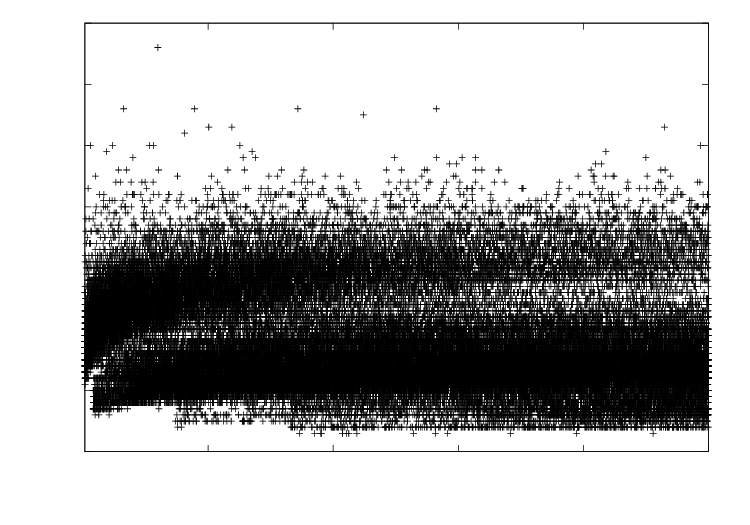}

    \input{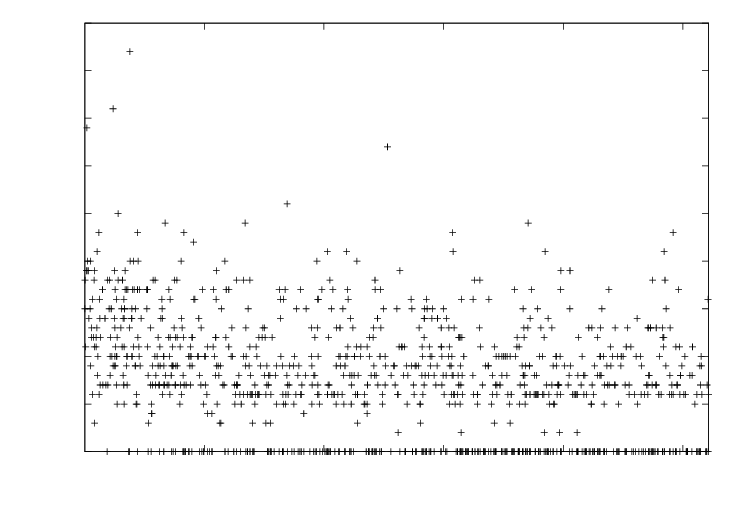}
    \caption{Maximum number of orbits checked during any iteration of Step~\ref{step:cosets} for all primes less than $1,000,000$ (top) and random primes less than $110,000,000$ (bottom).}%
    \label{fig:coset-checks}
\end{center}
\end{figure}

\begin{figure}[tbph]
\begin{center}
    \input{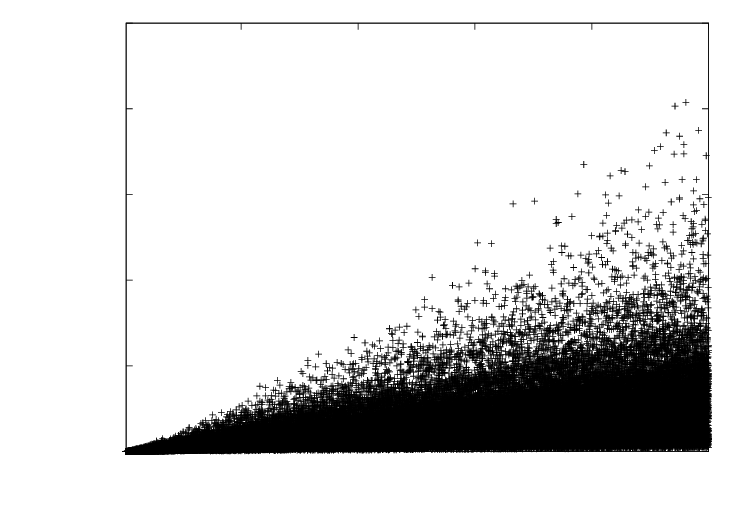}

    \input{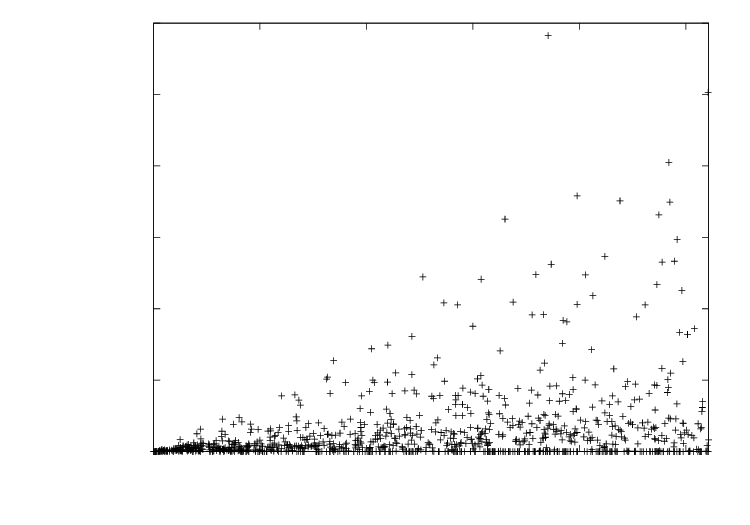}
    \caption{%
        Runtimes for Algorithm~\ref{alg:process} on all primes less than $1,000,000$ (top) and our random sample of primes less than $110,000,000$ (bottom). 
        Measured on an 11th Gen Intel Core i5-11320H processor.
    }%
    \label{fig:runtimes}
\end{center}
\end{figure}%

The actual time to run Algorithm~\ref{alg:process} depends heavily on both the endgame and middle game breakpoints.
A plot of a the breakpoints for our random primes beneath $110$ million, calculated using equation~\eqref{eq:endgame} and Algorithm~\ref{alg:middlegame}, is shown in Figure~\ref{fig:breakpoints}.
The first prime for which Algorithm~\ref{alg:middlegame} returns a middle game breakpoint is~$p = 1,328,247$.

\begin{figure}[tbph]
\begin{center}
    \input{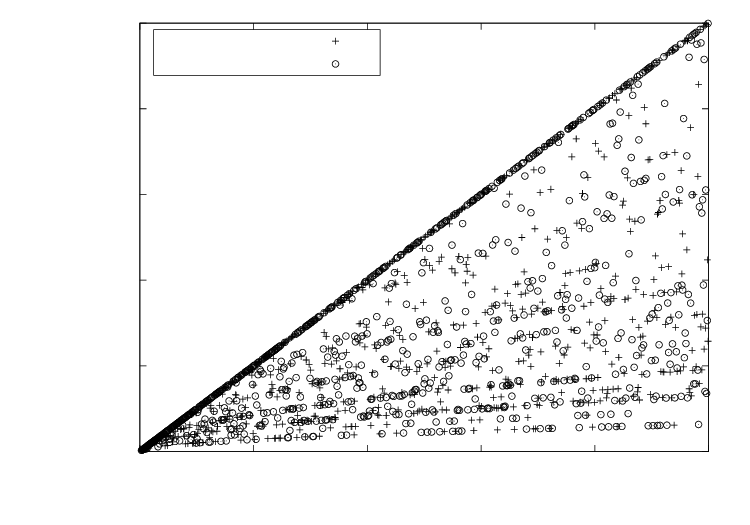}

    \input{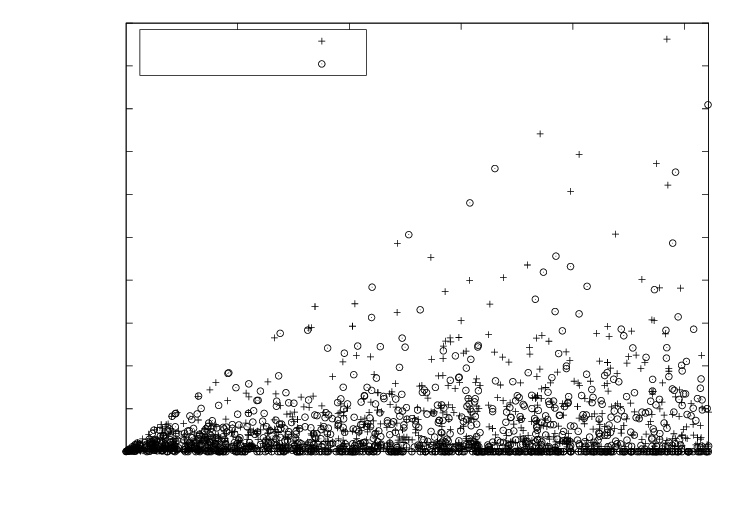}
    \caption{Breakpoints for all primes less than $1,000,000$ (top) and random primes less than $110,000,000$ (bottom).}%
    \label{fig:breakpoints}
\end{center}
\end{figure}%

\section*{Appendix: Notes on constant parameterization in Rust}
In this section, we discuss the choice to represent our functions and data structures as being constantly (i.e., determined at compile time) parameterized by a prime $P$.
This choice had major impacts on both the speed of compiling and running our algorithm described in Section 4.
While a similar approach is possible in other languages, including C using preprocessor macros, we found this technique particularly straightforward to implement in Rust.

Rust types, traits, and functions can be parameterized by types and constants.
When Rust code is compiled, it is monomorphized: every type, trait, and function with a parameterization possibly has a copy made per realization.
For an introduction to monomorphization in Rust, see~\cite{rust-language}.

In this project, we choose to parameterize almost all of our types by a prime \texttt{P}.
The compiler then applies numerous optimizations which can have dramatic effects on the runtimes.
As a demonstration, we compare two implementations of a simple program in Listing~\ref{lst:const-rust}.

\begin{lstlisting}[
      language=Rust
    , caption={Rust program showing the same function with and without constant parameterization.}
    , label={lst:const-rust}
    , frame=single
    , float=tph
]
fn foo(x: u64, p: u64) -> u64 {
    (x * 1_000_007) % p
}

fn bar<const P: u64>(x: u64) -> u64 {
    (x * 1_000_007) % P
}

fn main() {
    for i in 0..10 {
        foo(i, 7);
        bar::<7>(i);
    }
}
\end{lstlisting}%

In the unparametrized \texttt{foo}, the compiler translates the modulus operator into a \texttt{divl} assembler command.
When constantly parameterized, Rust compiles the \texttt{bar} function by optimizing away the division, replacing the modulus by~$7$ with a series of additions, subtractions, and shifts, along with a so-called ``magic number'' dependent on the modulus operand.
Most processors have single-cycle addition, subtraction, and shift operations, but division can be as much as twenty times slower.
The \texttt{bar} function is therefore significantly faster than \texttt{foo}, despite being compiled into more lines of assembler.
For an introduction to this technique, see~\cite[Chapter~10]{hackers-delight}.

However, this technique comes at a steep cost during compile times.
The compiler must generate bytecode for each of the realized monomorphized functions.
Additionally, because the compiler must statically know all type parameters and solve all trait constraints, the easiest way to iterate over many primes is via macro expansion, which happens very early on in the compilation process, leading to larger volumes of code for which to generate bytecode.

The Rust language team neither maintains an official list of compile time optimizations nor includes optimizations as part of their stability guarantees.
With that said, some standard optimizations in most modern compilers, including the Rust compiler version 1.74, include compile-time constant promotion and evaluation, branch elimination, and precomputed stack frame sizes.

To demonstrate the trade-offs of this method, we compiled and ran a demonstration using both methods.
We count the number of primes~$p$ in various ranges for which~$42$ is a quadratic nonresidue modulo~$p$.
Listings~\ref{lst:bench-no-const} shows this code when the prime is accepted as a function parameter, and Listings~\ref{lst:bench-with-const} shows the code with constant parameterization.
The \texttt{get\_primes()} method and \texttt{get\_primes!()} macro, respectively, return the primes within the given ranges.
Table~\ref{tbl:benchmarks} shows the results of our benchmarks, run on an 11th Gen Intel Core i5-11320H processor. 

\begin{table}[h]
\begin{center}
    \caption{Computation times for Listings~\ref{lst:bench-no-const}~and~\ref{lst:bench-with-const}.}%
    \label{tbl:benchmarks}
    \begin{tabular}{@{} cp{2cm}p{2cm}p{2cm}p{2cm} @{}}
        \toprule
        & \multicolumn{4}{c}{Compile times (miliseconds)} \\
        \cmidrule(lr){2-5}
        & \multicolumn{2}{c}{Without constants} & \multicolumn{2}{c}{With constants} \\
        \cmidrule(lr){2-3}\cmidrule(lr){4-5}
       $p$ domain &~$\mu$ &~$\sigma$ &~$\mu$ &~$\sigma$ \\
       $1,000,000$ to~$1,020,000$
            &~$251.3$
            &~$4.9$
            &~$475.3$
            &~$27.7$ \\
       $1,000,000$ to~$1,200,000$
            &~$252.2$
            &~$4.9$
            &~$41,037.3$
            &~$973.9$ \\
       $10,000,000$ to~$10,020,100$
            &~$250.4$
            &~$3.6$
            &~$9,146.0$
            &~$63.0$ \\
        & \multicolumn{4}{c}{Run times (miliseconds)} \\
        \cmidrule(lr){2-5}
        & \multicolumn{2}{c}{Without constants} & \multicolumn{2}{c}{With constants} \\
        \cmidrule(lr){2-3}\cmidrule(lr){4-5}
       $p$ domain &~$\mu$ &~$\sigma$ &~$\mu$ &~$\sigma$ \\
       $1,000,000$ to~$1,020,000$
            &~$42.0$
            &~$0.2$
            &~$1.0$
            &~$0.1$ \\
       $1,000,000$ to~$1,200,000$
            &~$51.9$
            &~$0.1$
            &~$1.0$
            &~$0.1$ \\
       $10,000,000$ to~$10,020,000$
            &~$520.3$
            &~$1.1$
            &~$1.0$
            &~$0.1$ \\
        \bottomrule
    \end{tabular}
\end{center}
\end{table}%

In our experience, the widespread use of constant parameterization allowed us to run our program on significantly larger primes than we otherwise reasonably could have, but was also a significant slowdown in the exhaustive search of all primes less than one million.
Further experimentation with these techniques should be a consideration in future projects utilizing the Rust programming language.

\begin{lstlisting}[
      language=Rust
    , caption={Benchmark program without constant parameterization.}
    , label={lst:bench-no-const}
    , frame=single
    , float=h
]
fn is_nonresidue(mut x: u64, p: u64) -> bool {
    let mut n = (p - 1) / 2;
    let mut y = 1;
    while n > 1 {
        if n % 2 == 1 {
            y = (y * x) % p;
        }
        x = (x * x) % p;
        n /= 2;
    }
    (x * y) % p == p - 1
}

fn main() {
    let mut total = 0;
    for p in get_primes() {
        if is_nonresidue(42, p) { total += 1; }
    }
    println!("{total}");
}
\end{lstlisting}%

\begin{lstlisting}[
      language=Rust
    , caption={Benchmark program with constant parameterization.}
    , label={lst:bench-with-const}
    , frame=single
    , float=h
]
fn is_nonresidue<const P: u64>(mut x: u64) -> u64 {
    let mut n = (P - 1) / 2;
    let mut y = 1;
    while n > 1 {
        if n % 2 == 1 {
            y = (y * x) % P;
        }
        x = (x * x) % P;
        n /= 2;
    }
    (x * y) % P == P - 1
}

fn main() {
    let mut total = 0;
    // Expands into an is_nonresidue call for each prime from the get_primes! macro.
    macro_rules! go {
        ($($P:literal),$(,)?) => {$(
            if is_nonresidue::<$P>(42) { total += 1; }
        )+};
    }
    get_primes!(go);
    println!("{total}");
}
\end{lstlisting}%

\FloatBarrier
\printbibliography

\end{document}